\documentclass[10pt]{article}
\usepackage{amsmath}
\usepackage{amsfonts, amssymb, amsthm}
\usepackage{epsfig}
\usepackage{color}
\usepackage{multicol}

\newlength{\hchng}
\newlength{\vchng}
\newtheorem{thm}{Theorem}[section]
\newtheorem{prop}[thm]{Proposition}
\newtheorem{cor}[thm]{Corollary}

\newtheorem{lemma}[thm]{Lemma}

\newtheorem{definition}[thm]{Definition}

\newtheorem{preremark}[thm]{Remark}
\newenvironment{remark}{\begin{preremark}\rm}{\medskip \end{preremark}}
\numberwithin{equation}{section}

\newcommand{\norm}[1]{\left\Vert#1\right\Vert}
\newcommand{\abs}[1]{\left\vert#1\right\vert}

\newcommand{\R}{\mathbb R}
\newcommand{\eps}{\varepsilon}

\newcommand{\grad} {\nabla}
\newcommand{\lap} {\triangle}

\newcommand{\dd} {\; \mathrm{d}}

\DeclareMathOperator*{\osc}{osc}
\DeclareMathOperator{\supp}{supp}
\DeclareMathOperator{\dv}{div}

\newcommand{\I} {\mathrm{I}}
\newcommand{\Q}{Q}

\title{On the differentiability of the solution to the Hamilton-Jacobi equation with critical fractional diffusion}
\author{Luis Silvestre}

\begin{document}
\maketitle

\begin{abstract}
We prove that the Hamilton Jacobi equation for an arbitrary Hamiltonian $H$ (locally Lipschitz but not necessarily convex) and fractional diffusion of order one (critical) has classical $C^{1,\alpha}$ solutions. The proof is achieved using a new H\"older estimate for solutions of advection diffusion equations of order one with bounded vector fields that are not necessarily divergence free.
\end{abstract}

\section{Introduction}

We consider the Hamilton-Jacobi equation with fractional diffusion
\begin{equation} \label{e:hj-fract}
u_t + H(\grad u) + (-\lap)^s u = 0.
\end{equation}
Where $H$ is a locally Lipschitz function. If the given initial data $u(x,0) = u_0(x)$ is smooth and $s > 1/2$, the smoothing effect of the fractional Laplacian term is stronger than the effect of the nonlinear term in small scales, and it is well known that the solution $u$ will remain smooth for positive time (see \cite{droniou2006fractal} and \cite{imbert2005non}), i.e. the problem is well posed in the classical sense.

The case $s =1/2$ is the most delicate because the two terms in the equation are of order one and their contributions are balanced at every scale. In this article we will show that the solution is also $C^{1,\alpha}$ smooth in this critical case as well.

For $s<1/2$ we can show that the equation develops singularities (discontinuities of the derivative) as a corollary of a construction of Kiselev, Nazarov and Shterenberg in \cite{kiselev-blow} for the fractional Burgers equation.

In \cite{droniou2006fractal}, it was shown that equation \eqref{e:hj-fract} has a unique viscosity solution in $W^{1,\infty}$ (i.e. Lipschitz) for any value of $s \in (0,1)$ if the initial data $u(-,0)$ is Lipschitz. The equation was also studied in \cite{imbert2005non} and \cite{karch2008fractal}. In \cite{droniou2003global}, the following one dimensional equation is studied:
\[ v_t + \partial_x f(v) + (-\lap)^s v = 0. \]
This equation corresponds to the one dimensional case of \eqref{e:hj-fract} with $f=H$ and $v = u_x$. They prove that the equation has a smooth solution if $s > 1/2$, but they explicitly leave the case $s = 1/2$ open, even in this one dimensional scenario.

The equation \eqref{e:hj-fract} arises in problems of optimal control of processes with $\alpha$-stable noise. It is a particular case of the more general first order Isaacs equation
\begin{equation} \label{e:hj-general}
u_t - \I u := u_t - \sup_i \inf_j \left( c^{ij} + b^{ij} \cdot \grad u + \int_{\R^n} \frac{u(x+y)-u(x)}{|y|^{n+1}} \ a^{ij} (y) \dd y \right) = 0.
\end{equation}
where $i$ and $j$ are two indexes ranging in arbitrary sets (controls), $c^{ij}$ is a family of constant scalars, $b^{ij}$ is a bounded family of vectors and the kernels $a^{ij} (y)$ satisfy
\begin{align*}
a^{ij} (y) &= a^{ij} (-y) \qquad \text{(symmetry)} \\
\lambda \leq a^{ij} (y) &\leq \Lambda \qquad \text{(uniform ellipticity)} \\
\end{align*}

The integral in \eqref{e:hj-general} makes sense in the principal value sense (thanks to the symmetry assumption on the kernels $a^{ij}$) as long as $u \in C^{1,\alpha}$ for some $\alpha>0$. The main result in this paper is that the solution $u$ is $C^{1,\alpha}$ in both time and space, which means that the equations \eqref{e:hj-fract} and \eqref{e:hj-general} are well posed in the classical sense. 

Note that for any Lipschitz function $H$, the equation \eqref{e:hj-fract} can be recovered from \eqref{e:hj-general} using a fixed constant $a$ and writing $H$ as 
\[ H(p) = \sup_i \inf_j (c^{ij} + b^{ij} \cdot p). \]
Such representation of $H$ always exists for some bounded family $b^{ij}$ and $c^{ij}$ if $H$ is Lipschitz. If $H$ is only \emph{locally} Lipschitz, equation \eqref{e:hj-fract} takes the form \eqref{e:hj-general} only a posteriori once we know that the solution $u$ is Lipschitz, so that $\grad u$ stays in a bounded domain where $H$ can be considered a Lipschitz function.

In this paper we study problem \eqref{e:hj-general}. The result holds for the particular case \eqref{e:hj-fract} in every situation where the solution $u$ is known to be Lipschitz.

The idea of the proof is the following. If $u$ is a solution of \eqref{e:hj-fract}, any directional derivative $v = u_e$ would satisfy the linearized equation
\begin{equation} \label{e:ad}
 v_t + w \cdot \grad v + (-\lap)^{1/2} v = 0
\end{equation}
for $w = DH(\grad u)$. 

For a critical advection diffusion equation like \eqref{e:ad}, it was proved recently by Caffarelli and Vasseur \cite{caffarelli2006drift} that the solutions become H\"older continuous if $\dv w =0$ and $w \in BMO$. Their method uses variational techniques in the style De Giorgi. Another proof was given by Kiselev and Nazarov in \cite{kiselev2002variation}. In this article we establish a new proof that does not require $w$ to be divergence free but it requires $w \in L^\infty$.  Our methods are non variational, as opposed to the De Giorgi style methods used in \cite{caffarelli2006drift}. This new H\"older estimate for an equation like \eqref{e:ad} (Theorem \ref{t:adv-diffusion}) is interesting by itself and provides a non variational counterpart to the result in \cite{caffarelli2006drift} and \cite{kiselev2002variation}. This is the key idea of this paper. Once we have it, the differentiability of the solution to \eqref{e:hj-fract} follows as a consequence.

Note that we have no information a priori about the vector field $w$ except that it is bounded by the Lipschitz norm of $H$: $|w| \leq A$. Moreover, since $w$ is not divergence free in general, we cannot make sense of \eqref{e:ad} in the distributional sense. The only information we have is that
\begin{align*}
v_t - A |\grad v| + (-\lap)^{1/2} v \leq 0 \\
v_t + A |\grad v| + (-\lap)^{1/2} v \geq 0
\end{align*}
which can be made sense of in the (Crandall-Lions) viscosity sense. So our H\"older estimates depend on these two inequalities only.

The key diminish of oscillation lemma is developed in section \ref{s:dimish-of-oscillation}. The heart of the argument is Lemma \ref{l:point-estimate}. In section \ref{s:c1a-regularity}, the results of section \ref{s:dimish-of-oscillation} are used to obtain the $C^{1,\alpha}$ estimates for $u$. In section \ref{s:rough-initial-data} we discuss the case when the initial data is not Lipschitz. In sections \ref{s:preliminaries}, \ref{s:maximal-operators} and \ref{s:lipschitz-continuity} we review the viscosity solution framework for Hamilton-Jacobi equations with fractional diffusion. In section \ref{s:supercritical-case}, we point out that in the supercritical case $s<1/2$, the solution to \eqref{e:hj-fract} can develop singularities as an immediate consequence of a construction of Kiselev and Nazarov \cite{kiselev-blow}.

The focus of this paper is on the $C^{1,\alpha}$ regularity and not on the existence and uniqueness of a Lipschitz viscosity solution (which was proved in \cite{droniou2006fractal} for \eqref{e:hj-fract}). Nevertheless we sketch most of the necessary proofs in order to make the paper more self contained.

\section{Preliminaries}
\label{s:preliminaries}
We start by recalling the definition of the fractional Laplacian. The operator $(-\lap)^s$ is defined quickly using the Fourier transform as $\widehat{(-\lap)^s u} (\xi) = |\xi|^{2s} \hat u(\xi)$. A more useful classical formula for the fractional Laplacian is
\[ (-\lap)^s u(x) = C_{n,s} \int_{\R^n} \frac{ u(x) - u(x+y) }{|y|^{n+2s}} \dd y. \]
The proof of this formula, as well as the computation of the precise constant $C_{n,s}$, can be found in the book of Landkof \cite{MR0214795}. 

If a function $u$ is $C^{1,\alpha}$ and bounded, the operator $(-\lap)^{1/2} u$ is well defined and $C^\alpha$. This follows from the identity $(-\lap)^{1/2} u = \sum R_i \partial_i u$ and the classical $C^\alpha$ estimates for the Riesz transform. For a general elliptic operator of order one, the same statement holds and is proved in the following proposition. The important point to make is that once we prove that \eqref{e:hj-fract} or \eqref{e:hj-general} have a $C^{1,\alpha}$ solution, this solution is classical.

\begin{prop} \label{p:classical}
Given any bounded symmetric function $a \leq \Lambda$ and $u \in C^{1,\alpha}$, the integro-differential operator
\[ L_a u = \int_{\R^n} \frac{u(x+y) - u(x)}{|y|^{n+1}} a(y) \dd y \]
is a $C^\alpha$ function whose $C^\alpha$ norm depends only on $\Lambda$, $\norm{u}_{C^{1,\alpha}}$ and the dimension $n$.

Consequently, any nonlinear operator of the form $\I u = \inf_i \sup_j c^{ij} + b^{ij} \cdot \grad u + L_{a^{ij}} u$ is also $C^\alpha$ if $u \in C^{1,\alpha}$ as long as the family of vectors $b^{ij}$ is bounded and the kernels $a^{ij}$ are uniformly bounded.
\end{prop}

\begin{proof}
We will write the operator $L_a u$ as a classical singular integral operator applied to the gradient $\grad u$. We write $u(x+y) - u(x)$ as the integral of $\grad u \cdot y$ along the segment from $x$ to $x+y$ and replace in the integral formula for $L_a u$:
\begin{align*}
L_a u(x) &=  \int_{\R^n} \int_0^1 \frac{\grad u(x+sy) \cdot y}{|y|^{n+1}} a(y) \dd s \dd y \\
&= \int_{\R^n} \grad u(x+z) \cdot \frac{k(z)}{|z|^n} \dd z
\end{align*}
where
\[ k(z) = \int_0^1 a \left( \frac z s \right) z \dd s. \]
Since $k$ is bounded and odd, then $L_a u \in C^\alpha$ if $\grad u \in C^\alpha$ by the classical H\"older estimates for singular integrals.
\end{proof}

The theory of viscosity solutions developed by Crandall and Lions is very suitable to study solutions of equations \eqref{e:hj-fract} or \eqref{e:hj-general}. Here is a standard definition of viscosity solution adapted to the specific case of equation \eqref{e:hj-general}.

\begin{definition}
An upper (lower) semicontinuous function $u$ is said to be a \emph{viscosity} subsolution (supersolution) of \eqref{e:hj-general} if every time a $C^2$ function $\varphi$ touches $u$ from above (below) at a point $(x,t)$ meaning that for some $r>0$
\begin{align*}
\varphi(x,t) &= u(x,t) \\
\varphi(y,s) &\geq u(y,s) \qquad \text{ for all $(y,s)$ such that $|x-y|<r$ and $t-r<s\leq t$.} \qquad (\text{resp. } \leq u(y,s))
\end{align*}
then if we construct the function $v$:
\[ v = \begin{cases}
\varphi & \text{in } U := B_r(x) \times (t-r,t+r)\\
u &\text{in } (\R^n \times [0,+\infty)) \setminus U
\end{cases} \]
then $v$ satisfies
\[ v_t(x,t) - \I v(x,t) \leq 0 \qquad (\text{resp. } \geq 0). \]

A \emph{viscosity solution} is a continuous function $u$ that is at the same time a subsolution and a supersolution.
\end{definition}

This definition is the straight forward adaptation of the corresponding definitions in \cite{barles2008second} and \cite{caffarelli2009regularity} for elliptic problems. 

The theory of viscosity solutions (See \cite{crandall1992user}) provides very general methods to obtain existence and uniqueness of solutions. Essentially, once a comparison principle is valid and barrier functions are constructed near the boundary of the equation, the existence follows by a standard Perron's method. 

Assuming that $u_0$ is continuous and bounded, either system \eqref{e:hj-fract} or \eqref{e:hj-general} has a unique bounded continuous viscosity solution $u$. This was proved in \cite{imbert2005non} and \cite{droniou2006fractal} for \eqref{e:hj-fract} and Lipschitz initial data using a vanishing viscosity approximation. The same ideas could also be applied to \eqref{e:hj-general}. In the appendix we sketch a proof using Perron's method.

\section{Maximal Operators}
\label{s:maximal-operators}
In this section we define the Pucci type extremal operators which are good over-estimators for the difference of two solutions of \eqref{e:hj-general}. These operators were defined originally in \cite{silvestre2006holder} and \cite{caffarelli2009regularity}.

\begin{definition}
Given $0<\lambda\leq \Lambda$, we write the nonlocal maximal and minimal operators $M^+$ and $M^-$ as
\begin{align*}
M^+ u(x) &= \frac 1 2 \int_{\R^n} \frac{\Lambda \delta u(x,y)^+ - \lambda \delta u(x,y)^-} {|y|^{n+1}} \dd y \\
M^- u(x) &= \frac 1 2 \int_{\R^n} \frac{\lambda \delta u(x,y)^+ - \Lambda \delta u(x,y)^-} {|y|^{n+1}} \dd y
\end{align*}
where $\delta u(x,y) := u(x+y)+u(x-y)-2u(x)$ and $x^+$ and $x^-$ stand for the positive and negative part of $x$ respectively.
\end{definition}

These operators are the extremals of all uniformly elliptic integro differential operators of order one in the sense that for any function $u$ which is $C^{1,\alpha}$ at the point $x$,
\begin{equation} \label{e:maximal-pucci}
M^+ u (x) = \sup_{\substack{ a(y) = a(-y) \\ \lambda \leq a(y) \leq \Lambda } } \int_{\R^n} \frac{u(x+y)-u(x)}{|y|^{n+1}} \ a (y) \dd y.
\end{equation}
The equality above can be seen easily by averaging the value of $y$ and $-y$ in the integral. Since $a(y) = a(-y)$,
\[ \int_{\R^n} \frac{u(x+y)-u(x)}{|y|^{n+1}} \ a(y) \dd y = \frac 1 2 \int_{\R^n} \frac{\delta u(x,y)}{|y|^{n+1}} \ a (y) \dd y.\]
Therefore, the choice of $a$ which would make the integral larger is to choose $a$ as large as possible where $\delta u(x,y)>0$ and as small as possible where $\delta u(x,y)<0$.

A formula like \eqref{e:maximal-pucci} also holds for $M^-$ replacing the $\sup$ by $\inf$. Since $M^+$ and $M^-$ are a supremum and an infimum respectively of linear integro-differential operators with bounded kernels $a$, the result of Proposition \ref{p:classical} applies. Thus $M^+ u$ and $M^- u$ will be $C^\alpha$ if $u \in C^{1,\alpha}$.

We use the extremal operators to write an equation for the difference of two solutions of \eqref{e:hj-general}.

\begin{lemma} [Equation for the difference of solutions] \label{l:S-class-for-difference}
Let $u$ be a viscosity subsolution and $v$ a viscosity supersolution of \eqref{e:hj-general}. Assume $|b^{ij}| \leq A$ for all indexes $i,j$. then the function $(u-v)$ is a viscosity subsolution of the equation
\begin{equation} \label{e:equation-for-difference}
(u-v)_t - A |\grad (u-v)| - M^+ (u-v) \leq 0.
\end{equation}

On the other hand, if $u$ be a viscosity supersolution and $v$ a viscosity subsolution of \eqref{e:hj-general}, $(u-v)$ is a viscosity supersolution of
\[ (u-v)_t + A |\grad (u-v)| - M^- (u-v) \geq 0. \]
\end{lemma}

In the case that the functions $u$ and $v$ are $C^{1,\alpha}$, this lemma is straightforward from the definition of the operators $M^+$ and $M^-$. When $u$ and $v$ are only semi continuous and satisfy the inequalities in the viscosity sense, the proof requires some work because the operators cannot be evaluated in the classical sense. These ideas are very well understood in the theory of viscosity solutions. We provide the proof in the appendix.

\begin{lemma}[Maximum principle] \label{l:maximum-principle}
Let $u$ be a bounded upper semicontinuous function which is a subsolution of the equation
\[ u_t - A |\grad u| - M^+ u \leq 0. \]
Then for every $t>0$, $\sup_x u(x,t) \leq \sup_x u(x,0)$.
\end{lemma}

\begin{proof}
Let $M > \sup_{x,t} u(x,t)$. We do the proof by contradiction, assume that for some $x_0 \in \R^n$ and $t_0>0$, $u(x_0,t_0) > \sup_x u(x,0)$. Let $H = M - u(x_0,t_0)$ and $\eps \ll (u(x_0,t_0) - \sup_x u(x,t))/(1+A)$.

Let $g = 1/(1+x^2)$ (any smooth bump function would suffice for this proof). Since $g$ is smooth, $M^+ g$ is continuous by proposition \ref{p:classical}. Moreover, we can rescale it $f(x) = g(\lambda x)$ with $\lambda$ sufficiently small so that $M^+ f (x) = \lambda M^+ g (\lambda x) < \eps$ and $|\grad f(x)|\leq \eps$ for all $x \in \R^n$. Therefore the function
\[ \varphi(x,t) = M - h f(x-x_0) + (1+A) h \eps t \]
is a supersolution of
\[ \varphi_t - A |\grad \varphi| - M^+ \varphi > 0, \]
for any value of $h > 0$. 

Note that for $h = 0$, $\varphi > u$ and for any $h>0$, $\varphi(x,0) > u(x,0)$. Moreover, for $h = H / (1-(1+A) \eps t)$, $\varphi(x_0,t_0) = u(x_0,t_0)$ and for $t> 1/(\eps(1+A))$, $\varphi > u$. Therefore there is a minimum value of $h$ where $\varphi$ does not stay strictly above $u$. For that value of $h$, $\varphi$ will touch $u$ from above at some point $(x_1,t_1)$ with $t_1>0$. But this is impossible by the definition of viscosity solution since $\varphi$ is a smooth supersolution of the equation.
\end{proof}

\begin{cor}[Comparison principle] \label{l:comparison-principle}
Let $u$ be a viscosity subsolution and $v$ a viscosity supersolution of \eqref{e:hj-general}. Assume $u(x,0) \leq v(x,0)$ for every $x \in \R^n$ and $u$ and $v$ are bounded. Then $u \leq v$ in $\R^n \times [0,+\infty)$.
\end{cor}

\begin{proof}
By Lemma \ref{l:S-class-for-difference}, $u-v$ satisfies the assumptions of Lemma \ref{l:maximum-principle}. Thus $u-v$ remains negative for all time and $u \leq v$.
\end{proof}

\section{Lipschitz continuity}
\label{s:lipschitz-continuity}
If the initial data $u_0$ is Lipschitz, we can provide a short proof that the same Lipschitz bound will be preserved by evolution of the equation \eqref{e:hj-general} using only the comparison principle. We show it in the next Lemma.

\begin{lemma} \label{l:lipschitz}
Assume $u$ is a bounded viscosity solution of \eqref{e:hj-general} and $u_0 = u(-,0)$ is a bounded Lipschitz function. Then $u$ is also Lipschitz in $x$ and for every $t \geq 0$, $\norm{u(-,t)}_{Lip} \leq \norm{u_0}_{Lip}$
\end{lemma}

\begin{proof}
The Lipschitz constant $\norm{u_0}_{Lip} = C$ is equivalent to the inequality
\[ u_0(x+y) \leq u_0(x) + C |y| \]
for every $x$ and $y$. Now, for every fixed $y$, the function $u(x+y,t)$ is a solution of \eqref{e:hj-general} with initial data $u_0(x+y)$. On the other hand, the function $u(x,t) + C|y|$ is a solution of \eqref{e:hj-general} with initial data $u_0(x) + C|y|$. By comparison principle (Corollary \ref{l:comparison-principle}), $u(x+y,t) \leq u(x,t) + C|y|$ for all $x$ and all $t$. Since this argument can be repeated for all $y$, we obtain that $u(-,t)$ is Lipschitz continuous, with Lipschitz constant $C$ for all $t \geq 0$.
\end{proof}

\begin{remark}
The Lemma is based on comparison principle only, so the same proof can be applied to prove that viscosity solutions to \eqref{e:hj-fract} remain Lipschitz continuous using only the comparison principle for \eqref{e:hj-fract}.
\end{remark}

\section{The diminish of oscillation lemma}
\label{s:dimish-of-oscillation}

In this section we prove the oscillation lemmas which will be used in section \ref{s:c1a-regularity} to obtain H\"older estimates. A somewhat simplified version of these Lemmas was used in \cite{CCS}. 

The following is the key lemma of the paper, which provides a pointwise estimate from an estimate in measure.

\begin{lemma}[point estimate] \label{l:point-estimate}
Let $u$ be an upper semi continuous function, $u \leq 1$ in $\R^n \times [-2,0]$ which satisfies the following inequality in the viscosity sense in $B_{2+2A} \times [-2,0]$.
\[ u_t - A |\grad u| - M^+ u \leq \eps_0 \]
Assume also that 
\[ \abs{\{u \leq 0\} \cap ( B_1 \times [-2 , -1]) } \geq \mu. \]
Then, if $\eps_0$ is small enough there is a $\theta>0$ such that $u \leq 1-\theta$ in $B_1 \times [-1 , 0]$.

\noindent (the maximal value of $\eps_0$ as well as the value of $\theta$ depend only on $A$, $\lambda$, $\Lambda$ and $n$)
\end{lemma}

\begin{proof}
We consider the following ODE and its solution $m: [-2 , 0] \to \R$:
\begin{equation} \label{e:ODEform}
\begin{aligned}
m(-2) &= 0, \\
m'(t) &= c_0 | \{x \in B_1: u(x,t) \leq 0\}| - C_1 m(t).
\end{aligned}
\end{equation}

The above ODE can be solved explicitly by the formula
\[ m(t) = \int_{-2}^t c_0 | \{x : u(x,s) \leq 0\} \cap B_1 | e^{-C_1
(t-s)} \dd s. \]

The strategy of the proof is to show that if $c_0$ is small and $C_1$ is large, then
$u \leq 1 - m(t) + 2 \eps_0$ in $B_1 \times [-1,0]$. 
Since for $t \in [- 1 ,0]$,
\[ m(t) \geq c_0 e^{- 2C_1  } |\{ u \leq 0 \} \cap B_1 \times
[-2  , -1 ] | \geq  c_0 e^{-2C_1} \mu.\]We can set $\theta = c_0 e^{- 2C_1 }\mu/2$ for $\eps_0$ small and obtain the result of the Lemma. 

Let $\beta : \R \to \R$ be a fixed smooth non increasing function such that $\beta(x)=1$ if $x \leq 1$ and $\beta(x)=0$ if $x \geq 2$.

Let $b(x,t) = \beta(|x|+ A t) = \beta (|x| - A |t|)$. As a function of $x$, $b(x,t)$ looks like a bump function for every fixed $t$. At those points where $b = 0$ (precisely where $|x| \geq 2 - A t = 2 + A |t|)$, $M^- b > 0$. Since $b$ is smooth, $M^- b$ is continuous by Proposition \ref{p:classical} and it remains positive for $b$ small enough. Thus, there is some constant $\beta_1$ such that $M^- b \geq 0$ if $b \leq \beta_1$.

Assume that $u(x,t) > 1 - m(t) + \eps_0 (2+t)$ for some point $(x,t)
\in B_1 \times [-1,0]$. We will arrive to a contradiction by looking at the maximum of the function
\[ w(x,t) = u(x,t) + m(t) b(x,t) - \eps_0 (2+t). \]
We are assuming that there is one point in $B_1 \times [-1,0]$ where $w(x,t) > 1$. Let $(x_0,t_0)$ be the point that realizes the maximum of $w$:
\[ w(x_0,t_0) = \max_{\R^n \times [-2,0]} w(x,t).\]
Note that this maximum is larger than one, and thus it must be realized in the support of $b$. So $|x_0| < 2 + A |t_0| \leq 2 + 2A$, and $(x_0,t_0)$ belongs to the domain of the equation.

Let $\varphi(x,t) := w(x_0,t_0) - m(t) b(x,t) + \eps_0 (2+t)$. Since $w$ realizes its maximum at $(x_0,t_0)$ then $\varphi$ touches $u$ from above at the point $(x_0,t_0)$. We can then use the definition of viscosity solution. For any neighborhood $U$ of $x_0$ we define.
\[ v(x,t) = \begin{cases}
u(x,t) & \text{if } x \notin U \\
\varphi(x,t) &\text{if } x \in U. 
\end{cases}\]
and we have 
\begin{equation} \label{e:equation-for-v}
v_t - A|\grad v| - M^+ v \leq \eps_0 \qquad \text{at } (x_0,t_0).
\end{equation}

We start by estimating $v_t(x_0,t_0)$.
\begin{equation} \label{e:bound-for-vt}
\begin{aligned}
v_t(x_0,t_0) &= -m'(t_0) b(x_0,t_0) - m(t_0) \partial_t b(x_0,t_0) + \eps_0 \\
&= -m'(t_0) b(x_0,t_0) + m(t_0) A |\grad b(x_0,t_0)| + \eps_0
\end{aligned}
\end{equation}

We also note that 
\begin{equation} \label{e:bound-for-grad}
|\grad v(x_0,t_0)| = m(t) |\grad b(x_0,t_0)|.
\end{equation}

The delicate part of the argument is to estimate $M^+ v$ correctly. Let us choose $U$ to be a tiny ball $B_r$ for some $r \ll 1$. The computations below are for fixed $t=t_0$, so we omit writing the time $t$ in order the keep the computations cleaner.

Since $u+m  b$ attains its maximum at $x_0$, $\delta u(x_0,y) \leq -m \, \delta b(x_0,y)$ for all $y$. Replacing this inequality in the formula for $M^+$ and $M^-$ we can easily obtain $M^+ v(x_0) \leq  - m \ M^- b(x_0)$, but this is not sharp enough. We need our estimate to take into account the measure of the set $\{u \leq 0\} \cap B_1$.

Let $y \in \R^n$ be such that $u(x_0+y)\leq 0$. We estimate $\delta u(x_0,y) + m \ \delta b(x_0,y)$.
\begin{align*}
\delta u(x_0,y) + m \ \delta b(x_0,y) &= \\
u(x_0+y) & + m \ b(x_0+y) + u(x_0-y) + m \ b(x_0-y) - 2 u(x_0) - 2 m \ b(x_0) \\
\intertext{Since $u+mb$ attains its maximum at $x_0$,}
&\leq u(x_0+y) + m \ b(x_0+y) - u(x_0) - m \ b(x_0) \\
\intertext{Since $u(x_0) + m \ b(x_0) = w(x_0,t_0) + \eps_0 (1+t_0) > 1$,}
&\leq 0 + m - 1 \\
\intertext{We choose $c_0$ small so that $m \leq 1/2$ and}
\delta u(x_0+y) + m \ \delta b(x_0+y) &\leq - \frac 1 2
\end{align*}

Now we estimate $M^+ v(x_0,t_0)$, we start writing the integral
\begin{align*}
M^+ v(x_0,t_0) &= -\frac{m(t_0)} 2 \int_{B_r} \frac{ \Lambda \delta b(x_0,y)^- - \lambda \delta b(x_0,y)^+}{|y|^{n+1}} \dd y + \int_{\R^n \setminus B_r} \frac{ \Lambda \delta u(x_0,y)^+ - \lambda \delta u(x_0,y)^-}{|y|^{n+1}} \dd y \\
\intertext{We estimate $\delta u(x_0,y)$ by above by $-\delta b(x_0,y)$ except at those points where $x_0+y$ is in the \emph{good} set $G := \{u \leq 0\} \cap B_1$ where we use that $\delta u + m \, \delta b \leq -1/2$}
&\leq -m(t_0) M^- b(x_0,t_0) + \int_{G \setminus B_r} \frac{ \Lambda \delta (u + mb) (x_0,y)^+ - \lambda \delta (u + mb) (x_0,y)^-}{|y|^{n+1}} \dd y \\
&\leq -m(t_0) M^- b(x_0,t_0) + \int_{G \setminus B_r} \frac{- \lambda/2}{|y|^{n+1}} \dd y \\
&\leq -m(t_0) M^- b(x_0,t_0) - c_0 |G \setminus B_r|
\end{align*}
for some universal constant $c_0$ (this is how $c_0$ is chosen in \eqref{e:ODEform}). Note that as $r \to 0$, the measure of the set $|G \setminus B_r|$ becomes arbitrarily close to $|G| = \{ x \in B_1 : u(x,t_0) \leq 0\}$.

We consider two cases and obtain a contradiction in both. Either $b(x_0,t_0) > \beta_1$ or $b(x_0,t_0) \leq \beta_1$.

Let us start with the latter. If $b(x_0,t_0) \leq \beta_1$, then $M^- b(x_0,t_0) \geq 0$, then
\begin{equation} \label{e:bound-case1}
M^+ v(x_0,t_0) \leq - c_0 |G \setminus B_r|
\end{equation}

Replacing \eqref{e:bound-for-vt}, \eqref{e:bound-for-grad} and \eqref{e:bound-case1} into \eqref{e:equation-for-v} we obtain
\[ \eps_0 \geq -m'(t_0) b(x_0,t_0) + \eps_0 + c_0 |G \setminus B_r| \]
but for any $C_1>0$ this will be a contradiction with \eqref{e:ODEform} by taking $r$ small enough.

Let us now analyze the case $b(x_0,t_0) > \beta_1$. Since $b$ is a smooth, compactly supported function, there is some constant $C$ (depending on $A$), such that $|M^- b| \leq C$.
Then we have the bound
\[
M^+ v(x_0,t_0) \leq -m(t_0) M^- b(x_0,t_0) + c_0 |G \setminus B_r| \leq -C m(t_0) + c_0 |G \setminus B_r|
\]
Therefore, replacing in \eqref{e:equation-for-v}, we obtain
\[ \eps_0 \geq -m'(t_0) b(x_0,t_0) + \eps_0 + c_0 |G \setminus B_r| - C m(t_0)\]

and we have
\begin{equation*} 
- m'(t_0) b(x_0,t_0) - C m(t_0) + c_0 |G \setminus B_r| \leq 0.
\end{equation*}

We replace the value of $m'(t_0)$ in the above inequality using
\eqref{e:ODEform} and let $r \to 0$ to obtain
\[  (C_1 b(x_0,t_0)- C) m(t_0) + c_0 (1- b(x_0,t_0)) |G| \leq 0. \]
Recalling that $b(x_0,t_0) \geq \beta_1$, we arrive at a contradiction if $C_1$ is
chosen large enough.
\end{proof}

In the next lemma we use the notation $\Q_r$ to denote the cylinder
\[ \Q_r := B_r \times [-r,0]. \]
We prefer to use this notation instead of $C_r$ because we reserve the letter $C$ for constants.

\begin{lemma}[diminish of oscillation] \label{l:diminish-of-oscillation}
Let $u$ be a bounded continuous function which satisfies the following two inequalities in the viscosity sense in $\Q_1$
\begin{align}
u_t - A |\grad u| - M^+ u &\leq 0 \\
u_t + A |\grad u| - M^- u &\geq 0
\end{align}
There are universal constants $\theta>0$ and $\alpha>0$ (depending only on $A$, the dimension $n$, and the ellipticity constants $\lambda$ and $\Lambda$) such that if
\begin{align*}
|u| &\leq 1 &&\text{in } \Q_1 := B_1 \times [-1,0] \\
|u(x,t)| &\leq 2|(4+4A) x|^\alpha - 1  &&\text{in } (\R^n \setminus B_1) \times [-1,0]
\end{align*}
then
\[ \osc_{\Q_{1/(4+4A)}} u \leq (1-\theta) \]
\end{lemma}

\begin{proof}
We consider the rescaled version of $u$:
\[ \tilde u(x,t) = u((4+4A) x, (4+4A) t). \]

The function $\tilde u$ will stay either positive of negative in half of the points in $B_1 \times [-2,-1]$ (in measure). More precisely, either $\{\tilde  u \leq 0 \} \cap (B_1 \times [-2,-1]) \geq |B_1|/2$ or $\{\tilde u \geq 0 \} \cap (B_1 \times [-2,-1]) \geq |B_1|/2$. Let us assume the former, otherwise we can repeat the proof with $-\tilde u$ instead of $\tilde u$.

We will conclude the proof as soon as we can apply Lemma \ref{l:point-estimate} to $\tilde u$. The only hypothesis we are missing is that $\tilde u$ is not bounded above by $1$. So we have to consider $v = \min(1,\tilde u)$ and estimate the error in the right hand side of the equation. We prove that if $\alpha$ is small enough, then $v$ satisfies
\[ v_t - A |\grad v| - M^+ v \leq \eps \]
for a small $\eps$ and we can apply Lemma \ref{l:point-estimate}.

Note that inside $\Q_{4+4A}$, $\tilde u \leq 1$, thus $v = \tilde u$. The error in the equation in $\Q_{2+2A}$ comes only from the tails of the integrals in the computation of $M^+ v$. Indeed, if $\varphi$ touches $v$ from above at a point $(x,t) \in \Q_{2+2A}$, then it also touches $\tilde u$ at the same point. Choosing a small neighborhood $U$ of $x$ and constructing
\begin{multicols}{2}
\[ w_1 = \begin{cases}
\varphi & \text{in } U \\
\tilde u & \text{outside } U
\end{cases} \]

\[ w_2 = \begin{cases}
\varphi & \text{in } U \\
v & \text{outside } U
\end{cases} \]
\end{multicols}
We see that $\partial_t w_1 - A |\grad w_1| - M^+ w_1 \leq 0$ at $(x,t)$. On the other hand $\partial_t w_1 - A |\grad w_1| = \partial_t w_2 - A |\grad w_2|$. So we estimate $M^+ w_1 - M^+ w_2$.
\begin{align*}
M^+ w_1(x,t) - M^+ w_2(x,t) 
&= \int_{x+y \notin B_{4+4A}} \Lambda (\tilde u(x+y) - 1)^+ \frac{\dd y}{|y|^{n+1}} \\
&\leq \int_{x+y \notin B_{4+4A}} \Lambda (2((4+4A)^2 |x+y|)^\alpha - 2)^+ \frac{\dd y}{|y|^{n+1}} \leq \eps_0
\end{align*}
if $\alpha$ is chosen small enough.

Thus, for any test function $\varphi$, $\partial_t w_2 - A |\grad w_2| - M^+ w_2 \leq \eps_0$ and $v$ satisfies that inequality in the viscosity sense. We can then apply Lemma \ref{l:point-estimate} to $v$ and conclude the proof.
\end{proof}

\section{$C^{1,\alpha}$ regularity}
\label{s:c1a-regularity}
\begin{thm} [H\"older continuity for advection-diffusion equations] \label{t:adv-diffusion}
Let $u$ be a bounded continuous function which satisfies the following two inequalities in the viscosity sense in $\R^n \times [0,t]$
\begin{align}
u_t - A |\grad u| - M^+ u &\leq 0 \label{e:ineq-1}\\
u_t + A |\grad u| - M^- u &\geq 0 \label{e:ineq-2}
\end{align}
There there is an $\alpha>0$ (depending only on $A$, the dimension $n$, and the ellipticity constants $\lambda$ and $\Lambda$) such that for every $t>0$ the function $u$ is $C^\alpha$. Moreover we have the estimate
\begin{equation} \label{e:calpha}
|u(x,t) - u(y,s)| \leq C \norm{u(-,0)}_{L^\infty} \frac{|x-y|^\alpha + |t-s|^\alpha}{t^\alpha} \qquad \text{for every $x,y \in \R^n$ and $0 \leq s \leq t$}
\end{equation}

Equivalently we can write the estimate as,
\[ \norm{u}_{C^\alpha(\R^n \times [t/2,t])} \leq \frac C {t^\alpha} \norm{u(-,0)}_{L^\infty} \]
\end{thm}

\begin{proof}
For any $(x_0,t_0)$, we consider the normalized function
\[ v(x,t) = \frac 1 {\norm{u}_{L^\infty}} \, u(x_0+t_0 x , t_0 ( t+1) ). \]

We prove the $C^\alpha$ estimate \eqref{e:calpha} at every point $(x_0,t_0)$ by proving a $C^\alpha$ estimate for $v$ at $(0,0)$. Note that since the $L^\infty$ norm of $u$ is non increasing in time, then $|v|\leq 1$. Moreover, $v$ is also a solution of the same equation \eqref{e:ineq-1} and \eqref{e:ineq-2}.

Let $r=1/(4+4A)$. The estimate follows as soon as we can prove
\begin{equation} \label{e:oscillation-in-cylinders}
\osc_{\Q_{r^k}} v \leq 2 r^{\alpha k} \qquad \text{for } k=0,1,2,3,\dots
\end{equation}

We will prove \eqref{e:oscillation-in-cylinders} by constructing two sequences $a_k$ and $b_k$ such that $a_k \leq  v \leq b_k$ in $\Q_{r^k}$, $b_k - a_k = 2 r^{\alpha k}$, $a_k$ is nondecreasing and $b_k$ is non increasing. We will construct the sequence inductively.

Since $|v| \leq 1$ everywhere, we can start by choosing some $a_0 \leq \inf v$ and $b_0 \geq \sup v$ so that $b_0 - a_0 = 2$. Assume we have constructed the sequences up to some value of $k$ and let us find $a_{k+1}$ and $b_{k+1}$.

We scale again by considering
\[ w(x,t) = (v(r^k x,r^k t) - (a_k+b_k)/2) r^{-\alpha k}. \]
Therefore we have 
\begin{align*}
|w| &\leq 1 && \text{ in } \Q_1 \\
|w| &\leq 2 r^{-\alpha k} - 1 &&  \text{in } \Q_{r^{-k}}  \qquad \text{therefore } |w(x,t)| \leq 2 |r^{-1} x|^\alpha - 1 \text{ where } |x|>1.
\end{align*}

If $\alpha$ is small enough, we can apply Lemma \ref{l:diminish-of-oscillation} to obtain $\osc_{\Q_r} \leq 1-\theta$. So, if $\alpha$ is chosen smaller than the $\alpha$ of Lemma \eqref{l:diminish-of-oscillation} and also so that $1-\theta \leq r^\alpha$, we have $\osc_{\Q_r} \leq r^\alpha$, which means $\osc_{\Q_{r^{k+1}}} \leq r^{\alpha(k+1)}$ so we can find $a_{k+1}$ and $b_{k+1}$ and we finish the proof.
\end{proof}

The following theorem is actually a corollary of Theorem \ref{t:adv-diffusion}, but we state it as a theorem since it is the main result of this paper.

\begin{thm} [$C^{1,\alpha}$ regularity for HJ with critical diffusion] \label{t:hj}
Let $u$ be a bounded continuous function which solves \eqref{e:hj-general} in $\R^n \times [0,t]$. Assume the initial data $u(-,0)$ is Lipschitz continuous. There is an $\alpha>0$ (depending only on $A$, the dimension $n$, and the ellipticity constants $\lambda$ and $\Lambda$) such that for every $t>0$ the function $u$ is $C^{1,\alpha}$ in both $x$ and $t$. Moreover
\[ \norm{u}_{C^{1,\alpha}(\R^n \times [t/2,t])} \leq \frac C {t^\alpha} \norm{u(-,0)}_{Lip} \]
\end{thm}

\begin{proof}
The proof follows by applying Theorem \ref{t:adv-diffusion} to incremental quotients of $u$. We start by proving the regularity in the space variable $x$. 

For any vector $e$, the incremental quotient
\[ v_e(x,t) = \frac{u(x+e,t) - u(x,t)}{|e|} \]
is bounded in $L^\infty$ by $\norm{u}_{Lip}$ and satisfies the hypothesis of Theorem \ref{t:adv-diffusion}. Therefore
\[ \norm{v_e}_{C^\alpha (\R^n \times [t/2,t])} \leq \frac C {t^\alpha} \norm{u(-,0)}_{Lip}. \]
uniformly in $e$. Thus $\grad_x u$ is $C^{\alpha}$ with an estimate
\[ \norm{\grad_x u}_{C^\alpha(\R^n \times [t/2,t])} \leq \frac C {t^\alpha} \norm{u(-,0)}_{Lip}. \]

Since $u$ is $C^{1,\alpha}$ in space, the operator right hand side in the equation \eqref{e:hj-general}, $Iu$ is bounded (and H\"older continuous). Therefore $u_t$ is bounded. Therefore, we can consider an incremental quotient in time
\[ w_h(x,t) = \frac{u(x,t+h) - u(x,t)} h, \]
and $w_h$ will be bounded in $L^\infty$ independently of $h$. Moreover, $w_h$ satisfies the hypothesis of Theorem \ref{t:adv-diffusion}, then $w_h$ is $C^\alpha$ independently of $h$, which implies that $u_t$ is H\"older continuous as well with an estimate
\[ \norm{u_t}_{C^\alpha(\R^n \times [t/2,t])} \leq \frac C {t^\alpha} \norm{u(-,0)}_{Lip}, \]
which finishes the proof.
\end{proof}

\section{Non smooth initial data}
\label{s:rough-initial-data}

We have proved in Theorem \ref{t:hj} that if the initial data $u(-,0)=u_0$ is Lipschitz, then the solutions $u$ immediately becomes $C^{1,\alpha}$ for $t>0$. In this section we will show that the Lipschitz condition on $u_0$ is not necessary.

For equation \eqref{e:hj-fract} the situation is somewhat different. If $H$ is globally Lipschitz, then it is a particular case of \eqref{e:hj-general} and the following theorem applies. If $H$ is only locally Lipschitz, then we must first show that $u$ is Lipschitz in order to apply our theorems for \eqref{e:hj-general}, and therefore the following result is not relevant.

\begin{thm} [$C^{1,\alpha}$ regularity for non Lipschitz initial data] \label{t:rough-initial-data}
Let $u$ be a bounded continuous function which solves \eqref{e:hj-general} in $\R^n \times [0,t]$. There is an $\alpha>0$ (depending only on $A$, the dimension $n$, and the ellipticity constants $\lambda$ and $\Lambda$) such that for every $t>0$ the function $u$ is $C^{1,\alpha}$ in both $x$ and $t$. Moreover
\[ \norm{u}_{C^{1,\alpha}(\R^n \times [t/2,t])} \leq \frac C {t^{1+\alpha}} \norm{u(-,0)}_{L^\infty} \]
\end{thm}

The proof of Theorem \ref{t:rough-initial-data} uses the following Lemma from \cite{cabre1995fully} (section 5.3) in order to improve the regularity estimates on $u$ applying Theorem \ref{t:adv-diffusion} repeatedly.

\begin{lemma} \label{l:telescopic-lemma}
Let $0<\alpha<1$, $0 < \beta \leq 1$ and $K>0$ be constants. Let $u \in L^\infty(\R^n)$ satisfy $\norm{u}_{L^\infty} \leq K$. Define, for $h \in \R^n$, with $0 < |h| \leq 1$, the incremental quotient
\[ v_{\beta,h}(x) = \frac {u(x+h)-u(x)}{|h|^\beta} \qquad \text{for } x \in \R^n. \]
Assume that $v_{\beta,h} \in C^\alpha(\R^n)$ with $\norm{v_{\beta,h}}_{C^\alpha(\R^n)} \leq K$ for any $0 < |h| \leq 1$. We then have
\begin{enumerate}
\item If $\alpha+\beta < 1$ then $u \in C^{\alpha+\beta}(\R^n)$ and $\norm{u}_{C^{\alpha+\beta}} \leq C K$.
\item If $\alpha+\beta \geq 1$ then $u \in Lip(\R^n)$ and $\norm{u}_{Lip} \leq C K$.
\end{enumerate}
where the constant $C$ depends only on $\alpha$ and $\beta$.
\end{lemma}

We use the Lemma above in order to probe Theorem \ref{t:rough-initial-data}.

\begin{proof}[Proof of Theorem \ref{t:rough-initial-data}]
All we must prove is that the function $u(-,t)$ becomes Lipschitz for any $t>0$. Then we can apply Theorem \ref{t:hj} to obtain the $C^{1,\alpha}$ estimate. 

Let $c_0 = \I 0$ (the operator $\I$ applied to the constant zero function, which returns a constant). We have that $v(x,t) = c_0 t$ is a particular solution to \eqref{e:hj-general}. We apply Theorem \ref{t:adv-diffusion} to $u-v$ to obtain that $u-v$ is $C^\alpha$ for $t>0$. Therefore $u$ becomes $C^\alpha$ for $t>0$.

Now, we apply iteratively Theorem \ref{t:adv-diffusion} to incremental quotients of $u$ of the form
\[ v_{\beta,h}(x) = \frac {u(x+h)-u(x)}{|h|^\beta}. \]
We start with $\beta=\alpha$. Since $u \in C^\alpha$, then $v_{\beta,h}$ is bounded in $L^\infty$ for any $t>0$ independently of $h$. From Theorem \ref{t:adv-diffusion}, we have that $v_{\beta,h}$ becomes uniformly $C^\alpha$ for $t>0$. But then from Lemma \ref{l:telescopic-lemma}, $u \in C^{2\alpha}$. We repeat this procedure to obtain $u \in C^\beta$ for $\beta = 2\alpha, 3 \alpha, 4 \alpha, \dots$, until we reach the $k$th step when $k\beta > 1$ and we obtain $u \in Lip$. Thus we can apply Theorem \ref{t:hj} and finish the proof.
\end{proof}

\begin{remark}
Under some special assumptions on $H$, it may be possible to prove that solutions of \eqref{e:hj-fract} whose initial data is only uniformly continuous become Lipschitz for $t>0$. This is well known for the classical Hamilton-Jacobi equation without fractional diffusion (see \cite{MR814955}).
\end{remark}

\section{The supercritical case. Non differentiability.}
\label{s:supercritical-case}

We can make an example of a non differentiable solution of \eqref{e:hj-general} with smooth initial data for any $s < 1/2$ in one dimension for $H(p) = |p|^2$ as a corollary of a result of Kiselev and Nazarov \cite{kiselev-blow}.

\begin{thm}
For any $s<1/2$ there is one smooth function $u_0 : \R \to \R$ such that the equation
\[ u_t + \frac{|u_x|^2 } 2 + (-\lap)^s u = 0 \]
does not have a global in time $C^1$ solution with $u(-,0)=u_0$.
\end{thm}

\begin{proof}
The derivative $v=u_x$ satisfies the fractional Burgers equation
\[ v_t + v \, v_x + (-\lap)^s v = 0. \] 
In \cite{kiselev-blow}, it is shown that this equation can develop shocks for any $s<1/2$.
\end{proof}

\section*{Appendix}

\subsection*{Proof of Lemma \ref{l:S-class-for-difference}}
In this appendix we prove Lemma \ref{l:S-class-for-difference}. We use Jensen's sup-convolutions idea \cite{jensen1988maximum}, which is a standard method in the theory of viscosity solutions.

The definition of viscosity solution satisfies the following important stability condition with respect to half relaxed limits (See \cite{caffarelli2009regularity} or \cite{crandall1992user}).

\begin{prop} \label{p:gamma-limits}
Let $v_k$ be a sequence of viscosity supersolutions of \eqref{e:hj-general}. Let $v_*$ be the lower half relaxed limit (equivalently the $\Gamma$-limit):
\[ v_*(x,t) = \liminf_{\substack{ k \to +\infty \\ y \to x \\ s \to t}} v_k(y,s). \]
Then $v_*$ is also a supersolution of \eqref{e:hj-general}.

If $u_k$ be a sequence of viscosity subsolutions of \eqref{e:hj-general}. Let $u^*$ be the upper half relaxed limit (equivalently $-u^*$ is the $\Gamma$-limit of $-u_k$):
\[ u^*(x,t) = \limsup_{\substack{ k \to +\infty \\ y \to x \\ s \to t}} u_k(y,s). \]
Then $u^*$ is also a subsolution of \eqref{e:hj-general}.
\end{prop}

We start by recalling the definition of inf and sup convolutions.
\begin{definition}
Given an upper semicontinuous function $u$, the sup-convolution $u^\eps$ is defined as
\[ u^\eps(x,t) = \sup_{y \in \R^n, s \in [0,+\infty)} u(y,s) - \frac 1 \eps (|x-y|^2 + |t-s|^2). \]
Conversely, for a lower semicontinuous function $v$, the inf-convolution $v_\eps$ is defined as
\[ v_\eps(x,t) = \inf_{y \in \R^n, s \in [0,+\infty)} v(y,s) + \frac 1 \eps (|x-y|^2 + |t-s|^2). \]
\end{definition}

The importance of this construction is that it regularizes the function $u$ making it semiconvex in the case of $u^\eps$ and semi concave in the case $u_\eps$, and therefore Lipschitz, while preserving the condition of sub or supersolution as it is pointed out by the following standard lemma.

\begin{lemma} \label{l:supconv}
If $u(x,t)$ is a viscosity subsolution of \eqref{e:hj-general}, then so is $u^\eps$.

If $v(x,t)$ is a viscosity supersolution of \eqref{e:hj-general}, then so is $v_\eps$.
\end{lemma}

\begin{proof}
The proof follows by observing that $u^\eps$ is a supremum of translations of $u$ and the equation \eqref{e:hj-general} is translation invariant. The same idea applies to $v_\eps$.
\end{proof}

In the next proposition we point out, without a proof, the standard properties of inf and sup convolution (See for example \cite{crandall1992user} for the proofs and further discussion).

\begin{prop} \label{p:supconv}
\begin{itemize}
\item The function $u^\eps$ is semiconvex in the sense that every point has a tangent paraboloid from below with opening $1/\eps$.
\item The function $v_\eps$ is semi concave in the sense that every point has a tangent paraboloid from above with opening $1/\eps$.
\item The functions $u^\eps$ and $v_\eps$ are Lipschitz, with Lipschitz constant bounded by $\frac 1 \eps \norm{u}_{L^\infty}^{1/2}$.
\item $u^\eps$ converges to $u$ and $v_\eps$ converges to $v$ in the half relaxed sense (i.e. $v_\eps$ $\Gamma$-converges to $v$ and $-u^\eps$ $\Gamma$-converges to $-u$) as $\eps \to 0$.
\end{itemize}
\end{prop}

Now we are ready to write the proof of Lemma \ref{l:S-class-for-difference} in the same fashion as the corresponding theorem in \cite{caffarelli2009regularity}.

\begin{proof}[Proof of Lemma \ref{l:S-class-for-difference}]
By Lemma \ref{l:supconv}, $u^\eps$ is also a subsolution of \eqref{e:hj-general} and $v_\eps$ a supersolution. Moreover, by Proposition \ref{p:supconv}, $u^\eps \to u$ and $v_\eps \to v$ in the half relaxed sense. By the stability of viscosity solutions under half relaxed limits, it is enough to show that $w = (u^\eps - v_\eps)$ is a subsolution of \eqref{e:equation-for-difference}.

Let $\varphi$ be a smooth function touching $(u^\eps - v_\eps)$ from above at the point $(x,t)$ (with $t>0$). Note that for any $\eps>0$, both functions $u^\eps$ and $- v_\eps$ are semiconvex, which means that they have a tangent paraboloid from below of opening $1/\eps$. Since the smooth function $\varphi$ touches $(u^\eps - v_\eps)$ by above at the point $(x,t)$, then both $u^\eps$ and $-v_\eps$ must be $C^{1,1}$ at the point $(x,t)$ (meaning that there are tangent paraboloids from both sides). But this means that $\partial_t u^*$, $\partial_t v_*$, $\grad u^*$ and $\grad v_*$ are well defined classically at the point $(x,t)$. Moreover it also means that we can evaluate $\I u^\eps(x,t)$ and $\I v_\eps(x,t)$ in the classical sense. But then we compute directly at the point $(x,t)$,
\begin{align*}
(u^\eps-v_\eps)_t - A |\grad (u^\eps-v_\eps)| - M^+ (u^\eps-v_\eps) \leq \partial_t u^* - \I u^* - \partial_t v_* + \I v_* \leq 0 
\end{align*}
which clearly implies that also 
\[
\varphi_t(x,t) - A |\grad \varphi(x,t)| - M^+ \varphi(x,t) \leq 0
\]
since at the point $(x,t)$, $\varphi_t = (u^\eps-v_\eps)_t$, $\grad \varphi = \grad (u^\eps-v_\eps)$ and $M^+ \varphi \geq M^+ (u^\eps-v_\eps)$, all in the classical sense.

This proves that $u-v$ satisfies \eqref{e:equation-for-difference}. The other inequality follows in a similar way.
\end{proof}

\subsection*{Comparison principle for \eqref{e:hj-fract}}
Given a locally Lipschitz kernel $H$, we can apply Lemma \ref{l:S-class-for-difference} to a solution of \eqref{e:hj-fract} only a posteriori if we know that the viscosity solution $u$ is Lipschitz, so that we can restrict $H$ to a bounded domain. The Lipschitz continuity of the solution $u$ of \eqref{e:hj-fract} is a simple consequence of comparison principle if the initial data is already Lipschitz (Lemma \ref{l:lipschitz}). However, we need to prove the comparison principle for equation \eqref{e:hj-fract} without using Lemma \ref{l:S-class-for-difference}. This can be done by applying the idea of the proof of Lemma \ref{l:maximum-principle} directly to the difference of sup and inf convolutions. However one must assume that the inequality $u(-,0) \leq v(-,0)$ holds in a somewhat uniform way. More precisely, for every $\eps>0$, there is a $\delta>0$ such that
\begin{equation} \label{e:uniform-inequality}
 u(x,t) \leq v(y,s) + \eps \qquad \text{if } |x-y|<\delta \text{ and } |t|,|s| < \delta.
\end{equation}
This is the case if for example $u(-,0) \leq v(-,0)$ and both $u$ and $v$ are uniformly continuous on $\R^n \times \{0\}$. Note that this condition implies that there is a modulus of continuity $\omega$ such that
\[ u^\eps(x,0) \leq v_\eps(x,0) + \omega(\eps) \qquad \text{for all } x \in \R^n. \]
Where $u^\eps$ and $v_\eps$ are the sup and inf convolutions respectively.

\begin{prop} \label{p:comparison-principle}
Let $u$ be a subsolution of \eqref{e:hj-fract} and $v$ be a supersolution of \eqref{e:hj-fract}. Assume that for every $\eps>0$, there is a $\delta>0$ such that \eqref{e:uniform-inequality} holds. Then $u \leq v$ in $\R^n \times [0,+\infty)$.
\end{prop}

\begin{proof}
Let $u^\eps$ and $v_\eps$ be the sup and inf convolutions, which are a subsolution and a supersolution respectively. Since \eqref{e:uniform-inequality} holds, there is a modulus of continuity $\omega$ such that
\[ u^\eps(x,0) \leq v_\eps(x,0) + \omega(\eps) \qquad \text{for all } x \in \R^n. \]
We will prove that $u^\eps - v_\eps \leq \omega(\eps)$ in $\R^n \times [0,+\infty)$. The Proposition follows by passing to the limit as $\eps \to 0$.

Assume that $\sup u^\eps - v_\eps - \omega(\eps) = \sigma > 0$. We cannot assure that the supremum is realized at any point, however, for any $\delta>0$, there is some point $(x_0,t_0)$ (with $t_0>0$) such that $u^\eps(x_0,t_0) - v_\eps(x_0,t_0) - \omega(\eps) > \sigma - \delta$.

Let $g$ be a smooth bump function, for example $g(x) = 1/(1+|x|^2)$. We proceed like in the proof of Lemma \ref{l:maximum-principle}, we look for the maximum of the function 
\[w = u^\eps - v_\eps + 2\delta g(\lambda(x-x_0)) - \frac {\delta}{t_0} t\] 
for $\lambda \ll 1/t_0$. If $\delta$ is small enough, the maximum must be realized at some point $(x_1,t_1)$ with $t_1>0$ since $w(x_0,t_0)> \sigma + \omega(\eps)$ but $w(x,t) \leq \omega(\eps)$ for $x$ large or $t$ large and $w \leq \omega(\eps) + 2\delta$ on $t=0$.

Like in the proof of Lemma \ref{l:S-class-for-difference}, at the point $(x_1,t_1)$ where the maximum is realized, the functions $u^\eps$ and $v_\eps$ are $C^{1,1}$ (in the sense that they have a tangent paraboloid from above and below). So, the equation can be evaluated in the classical sense at that point and we obtain
\begin{align*}
\partial_t u^\eps &= \partial_t v_\eps + \frac \delta t_0, \\
\grad u^\eps &= \grad v_\eps - \delta \lambda \grad g(\lambda(x_1-x_0)),\\
(-\lap)^{1/2} u^\eps &= (-\lap)^{1/2} v_\eps - \delta \lambda (-\lap)^{1/2} g(\lambda(x_1-x_0)).
\end{align*}

From the Lipschitz estimate on $u^\eps$ and $v_\eps$, we know that $|\grad u^\eps|, |\grad v_\eps| \leq C / \eps$. Let $A$ be the Lipschitz constant of $H$ in the ball of radius $2C / \eps$. Then we have
\begin{align*}
0 &\geq \partial_t u^\eps + H(\grad u^\eps) + (-\lap)^{1/2} u^\eps \\
&\geq \partial_t v_\eps + \delta/t_0 + H(\grad v_\eps) + A \delta \lambda \norm{\grad g}_{L^\infty} + \delta \lambda \norm{(-\lap)^{1/2} g}_{L^\infty} \\
&\geq \frac \delta t_0 - A \delta \lambda \norm{\grad g}_{L^\infty} - \delta \lambda \norm{(-\lap)^{1/2} g}_{L^\infty}
\end{align*}

But this is a contradiction if $\lambda$ was chosen much smaller than $1/t_0$, and we finish the proof.
\end{proof}

\subsection*{Existence of solutions}
We now sketch the proof of existence of a viscosity solution to either \eqref{e:hj-general} or \eqref{e:hj-fract} by Perron's method. We write the proof for \eqref{e:hj-fract}, but the same proof would work for \eqref{e:hj-general}. We choose to write it for \eqref{e:hj-fract} because that case is slightly more difficult since the comparison principle requires the inequality at time $t=0$ to hold uniformly as in \eqref{e:uniform-inequality}.

Let $u_0$ be a uniformly continuous function in $\R^n$. We will prove that there exists a continuous function $u : \R^n \times [0,+\infty)$ that solves \eqref{e:hj-fract}, and $u(-,0)=u_0$.

Perron's method consists in taking the infimum (or relaxed infimum) of the family of all supersolutions of the equation. There are standard methods, using the comparison principle, to prove that this infimum is a continuous viscosity solution. But there is an extra difficulty for every particular equation in constructing the appropriate barrier functions  in order to prove that the infimum of all supersolutions is continuous at $t=0$ and the initial condition $u(-,0)=u_0$ is satisfied.

Let $\mathcal{U}$ be the set of all supersolutions $u$ such that there is some modulus of continuity $\omega$ so that for every $x,y \in \R^n$ and $t>0$,
\begin{equation} \label{e:unif-boundary-condition}
 u(y,t) \geq u_0(x) - \omega(|y-x|+t).
\end{equation}

We start by constructing appropriate barriers to show that the set $\mathcal{U}$ is non empty and bounded below.

Let $b$ be a smooth bump function such that
$b(0)=1$, $b \leq 1$ and $\supp b = B_1$. Then, depending on the modulus of continuity of $u_0$, for every $\eps>0$ and $x_0 \in \R^n$, there is a $\delta>0$ so that 
\begin{align*}
U_0(x) := b((x-x_0)/\delta) u_0(x_0) + (1-b((x-x_0)/\delta)) \norm{u_0}_{L^\infty} + \eps &\geq u_0(x) \\
L_0(x) := b((x-x_0)/\delta) u_0(x_0) - (1-b((x-x_0)/\delta)) \norm{u_0}_{L^\infty} - \eps &\leq u_0(x)
\end{align*}

Since $U_0$ and $L_0$ are smooth functions, $|\grad U_0|$, $|(-\lap)^{1/2} U_0|$, $|\grad L_0|$ and $|(-\lap)^{1/2} L_0|$ are bounded by some constant $C$. So we can construct a supersolution and a subsolution respectively by
\begin{align*}
U(x) := U_0(x) + Ct \\
L(x) := L_0(x) - Ct
\end{align*}

Note that both $U$ and $L$ are uniformly continuous, so we can apply the comparison principle of Proposition \ref{p:comparison-principle}. By comparison principle, every $u \in \mathcal{U}$ satisfies $u \geq L$ for all lower barriers $L$ (for all $x_0 \in \R^n$ and $\eps >0$). Moreover, $\mathcal{U}$ is not empty since every upper barrier $U$ belongs to $\mathcal{U}$.

Let $u_*$ be the following function:
\[ u_*(x,t) = \liminf_{r \to 0} \inf_{\substack{|x-y|<r \\ |t-s|<r}} \inf_{u \in \mathcal U} u(y,s) \]

For every $x_0 \in \R^n$ and $\eps>0$, $L(x,t) \leq u_*(x,t) \leq U(x,t)$, since $U \in \mathcal{U}$ and for every $u \in \mathcal{U}$, $u \geq L$ by the comparison principle (Proposition \ref{p:comparison-principle}). Therefore, $u_*$ is uniformly continuous on $\R^n \times \{0\}$ and $u_*(-,0) = u_0$.

It can be shown (see \cite{crandall1992user} for the general method) that $u_*$ is a continuous viscosity solution of the equation.

\section{Acknowledgments}

Luis Silvestre was partially supported by an NSF grant and the Sloan fellowship.

\bibliographystyle{plain}   
\bibliography{chj}             
\index{Bibliography@\emph{Bibliography}}%

\end{document}